\documentclass{amsart}
\usepackage{amsthm}
\usepackage{amsfonts}
\usepackage{amssymb}
\usepackage{amsmath}
\usepackage{mathrsfs}
\usepackage{csquotes}
\begin{document}

\large


\theoremstyle{plain}
\newtheorem{theorem}{Theorem}[section]

\newtheorem{prop}[theorem]{Proposition}

\newtheorem{lemma}{Lemma}[section]
\newtheorem{corol}{Corollary}[theorem]

\theoremstyle{definition}
\newtheorem{definition}{Definition}[section]
\newtheorem{remark}{\textnormal{\textbf{Remark}}}
\newtheorem{proposition}{\textnormal{\textbf{Proposition}}}
\newtheorem*{acknowledgement}{\textnormal{\textbf{Acknowledgement}}}
\theoremstyle{remark}
\newtheorem{example}{Example}
\numberwithin{equation}{section}

\title[When degree of roughness is a neighborhood over locally solid Riesz spaces]{When degree of roughness is a neighborhood over locally solid Riesz spaces}

\author[Sanjoy Ghosal and Sourav Mandal]{Sanjoy Ghosal and Sourav Mandal}

\thanks{Department of Mathematics, University of North Bengal, Raja Rammohunpur, Darjeeling-734013, West Bengal, India.}
 \thanks{\emph{E-mail}: sanjoykumarghosal@nbu.ac.in; sanjoyghosalju@gmail.com (S. Ghosal); sourav7478@gmail.com (S. Mandal)}
\date{}


\date{}

\thanks{Research of the second author is supported by UGC Research, HRDG, India.}

\setcounter {page}{1}

\maketitle

\begin{abstract}
In this paper we introduce the notion of rough weighted $\mathcal{I}_\tau$-limit points set and weighted $\mathcal{I}_\tau$-cluster points set in a locally solid Riesz space which are more generalized version of rough weighted $\mathcal{I}$-limit points set and weighted $\mathcal{I}$-cluster points set in a $\theta$-metric space respectively. Successively to compare with the following important results of Fridy [Proc. Amer. Math. Soc. {118} (4) (1993), 1187-1192] and Das [Topology Appl. {159} (10-11) (2012), 2621-2626], respectively be stated as
\begin{description}
  \item[(i)] Any number sequence $x=\{x_{n}\}_{n\in \mathbb{N}},$ the statistical cluster points set of $x$ is closed,

  \item[(ii)] In a topological space the $\mathcal{I}$-cluster points set is closed,
\end{description}
 we show that in general, the weighted $\mathcal{I}_\tau$-cluster points set in a locally solid Riesz space may not be closed. The resulting summability method unfollows some previous results in the direction of research works of Aytar [Numer. Funct. Anal. Optim. {29} (3-4) (2008) 291-303], D$\ddot{\mbox{u}}$ndar [Numer. Funct. Anal. Optim. {37} (4) (2016) 480-491], Ghosal [Math. Slovaca {70} (3) (2020)  667-680] and Sava\c{s}, Et [Period. Math. Hungar.  71 (2015) 135-145].

\end{abstract}

\section{{Introduction}}

The idea of convergence of a sequence in a norm linear space $(X, ||\textbf{.}||)$ had been extended to rough convergence first by Phu \cite{Phu1} as follows: Let $r$ be a \emph{\textbf{non-negative real number}}, a sequence  $x=\{x_n\}_{n\in \mathbb{N}}$ in $X$ is said to be rough convergent to $x_*$ w.r.t the roughness of degree $r,$ denoted by $x_n\xrightarrow{r}x_*$  provided that
$$~\mbox{for any}~ \varepsilon > 0~~\exists~ n_{\varepsilon} \in \mathbb{N}: n\geq n_{\varepsilon} ~\Rightarrow~ ||x_n - x_*|| \leq  r+ \varepsilon.$$ The set $LIM^r x=\left\{x_*\in X: x_n\xrightarrow{r}x_*\right\}$ is called the $r$-limit set of the sequence $x=\{x_n\}_{n\in \mathbb{N}}.$ Phu studied the set $LIM^r x$ of all such points and showed that this set is bounded, closed and convex.\\

 In the year 2013, the idea of rough $\mathcal{I}$-convergence was introduced by Pal et al. \cite{Pal} as a generalization of rough convergence \cite{Phu1,Phu2}, statistical convergence \cite{Fast,St}, rough statistical convergence \cite{Aytar2} and $\mathcal{I}$-convergence \cite{KMSS0,KSW} which is based on the structure of the ideal $\mathcal{I}$ of subsets of the set $\mathbb{N}$ as: Let $r$ be a \emph{\textbf{non-negative real number}}. A sequence  $x=\{x_n\}_{n\in \mathbb{N}}$ in $X$ is said to be rough $\mathcal{I}$-convergent to $x_*,$ denoted by $x_n\xrightarrow[r]{\mathcal{I}}x_*,$ provided for any $\varepsilon>0$ the set
 $$\left\{n\in \mathbb{N}: ||x_n-x_*||\geq r+\varepsilon \right\}\in \mathcal{I}.$$  The basic properties of this interesting concept were studied by Pal et al. \cite{Pal} in an arbitrary norm linear space.  We could follow references \cite{Aytar1,Dundar2016,F1,KMSS,Salat,Savas} related to the concepts of statistical convergence, rough convergence and others.\\

  One of the most impressive generalizations of the notion of rough $\mathcal{I}$-convergence is the concept rough weighted $\mathcal{I}$-convergence. Motivated from the definitions of rough weighted statistical limit points set and weighted statistical cluster points set \cite{DGGS}, recently  Ghosal et al. \cite{Ghosal2020} introduced the notion of rough weighted $\mathcal{I}$-limit points set and  weighted $\mathcal{I}$-cluster points set on $\theta$-metric space $(X,d_\theta)$ by using the weighted sequence of real numbers $\{t_{n}\}_{n\in \mathbb{N}}$ (i.e., $t_{n}>\delta,$ for all $n\in \mathbb{N}$ for some positive real number $\delta$) as follows:

\begin{definition}\cite{Ghosal2020}    Let $r$ be a \emph{\textbf{non-negative real number}} and $\{t_{n}\}_{n\in \mathbb{N}}$ be a weighted sequence. A sequence $x=\{x_{n}\}_{n\in \mathbb{N}}$ in a $\theta$-metric space $X$ is said to be rough weighted $\mathcal{I}$-convergent to $x_* \in X$  w.r.t the roughness of degree $r$ if for every $\varepsilon>0$, $$\left \{n\in \mathbb{N}: t_nd_\theta(x_{n},x_*)\geq r+\varepsilon \right\}\in \mathcal{I}.$$ In this case we write $x_{n}\xrightarrow[r]{W\mathcal{I}} x_*.$  The set $W\mathcal{I}-LIM^{r}x=\left\{x_*\in X: x_{n}\xrightarrow[r]{W\mathcal{I}} x_*\right\}$ is called the rough weighted $\mathcal{I}$-limit set of the sequence $x=\{x_{n}\}_{n\in \mathbb{N}}$ with degree of roughness $r.$ The sequence $x=\{x_{n}\}_{n\in \mathbb{N}}$ is said to be rough weighted $\mathcal{I}$-convergent provided that $W\mathcal{I}-LIM^{r}x\neq \varnothing.$ Visit \cite{Karakaya,LR2} for more references related this topic. \end{definition}

\begin{definition}\cite{Ghosal2020} Let $\{t_{n}\}_{n\in \mathbb{N}}$ be a weighted sequence and $c^{*} \in X$ is called a weighted $\mathcal{I}$-cluster point of a sequence $x=\{x_n\}_{n\in \mathbb{N}}$ in a $\theta$-metric space $X$ if for every $\varepsilon>0,$
$$\{n\in \mathbb{N}: t_n d_\theta(x_n,c^{*})<\varepsilon\}\notin \mathcal{I}.$$
We denote the set of all weighted $\mathcal{I}$-cluster points of the sequence $x=\{x_n\}_{n\in \mathbb{N}}$ by $ W\mathcal{I}(\Gamma_x).$
\end{definition}

The notion of Riesz space was first introduced by Riesz \cite{Riesz} in 1928 and since then it has found several applications in measure theory, operator theory, optimization. It is well known that a topology on a vector space that makes the operations of addition and scalar multiplication continuous is called a linear topology and a vector space endowed with a linear topology is called a topological vector space. A Riesz space is an ordered vector space which is also a lattice, endowed with a linear topology. Further if it has a base consisting of solid sets at zero then it is known as a locally solid Riesz space. We briefly recall some of the basic notions in the theory of Riesz space and we refer readers to \cite{AA,Luxemburg,Roberts,Savas0} for more details.


\begin{definition} \cite{AP} Let $L$ be a real vector space and $\leq$ be a partial order on this space. $L$ is said to be an ordered vector space if it satisfies the following properties:\\
(i) If $x, y \in L$ and $y \leq x,$ then $y + z \leq x+ z$ for each $z\in L.$\\
(ii) If $x, y\in L$ and $y \leq x,$ then $\lambda y \leq \lambda x$ for each $\lambda \geq 0.$ \end{definition}

In addition, if $L$ is a lattice with respect to the partial ordering, $L$ is said to be a Riesz space (or a vector lattice).

For an element $x$ in a Riesz space $L$  the positive and negative parts of $x$ are defined by $x^{+} = x\vee \theta$ and $x^{-} = (-x)\vee\theta$ respectively.
The absolute value of $x$ by $|x| = x \vee(-x),$ where $\theta$ is the element zero of $L.$

A subset $S$ of a Riesz space $L$ is said to be solid if $y\in S$ and $|x| \leq |y|$ imply $x\in S.$

A topology $\tau$ on a real vector space $L$ that makes the addition and the scalar multiplication continuous is said to be a
linear topology, that is, the topology $\tau$ makes the functions
$$(x, y) \rightarrow x+ y~ (\mbox{from}~ (L\times L, \tau \times \tau)\rightarrow (L,\tau)),$$
$$(\lambda, x) \rightarrow \lambda x ~(\mbox{from}~ (\mathbb{R}\times L,\sigma \times \tau) \rightarrow (L,\tau))$$
continuous, where $\sigma$ is the usual topology on  $\mathbb{R}.$ In this case, the pair $(L, \tau)$ is called a topological vector space.

Every linear topology $\tau$ on a vector space $L$ has a base $\mathcal{N}$ for the neighborhoods of $\theta$ (zero) satisfying the following
properties:\\
(a) Each $V\in\mathcal{N}$ is a balanced set, that is, $\lambda x\in V$ holds for all $x \in V$ and every $\lambda \in \mathbb{R}$ with $|\lambda| \leq 1.$\\
(b) Each $V\in\mathcal{N}$ is an absorbing set, that is, for every $x\in L,$ there exists a $\lambda>0$ such that $\lambda x\in V. $\\
(c) For each $V\in \mathcal{N}$ there exists some $W\in\mathcal{N}$ with $W + W \subseteq V.$


 \begin{definition} \cite{AP} A linear topology $\tau$ on a Riesz space $L$ is said to be locally solid if $\tau$ has a base at zero consisting of solid sets. A locally solid Riesz space $(L, \tau)$ is a Riesz space $L$ equipped with a locally solid topology $\tau.$\end{definition}

The symbol $\mathcal{N}_{sol}$ will stand for a base at zero consisting of solid sets and satisfying the properties (a), (b) and (c) in a locally solid topology.


\begin{definition}  \cite{Hong} A subset $A$ of a Riesz space $(L,\tau)$ is said to be topologically bounded or $\tau$-bounded if for every neighborhood $U$ of zero there exists some $\lambda>0$ such that $A\subset \lambda U.$ If $A$ is not $\tau$-bounded then it is called $\tau$-unbounded.\end{definition}


\begin{definition} \cite{Hong} A subset $B$ of a Riesz space $(L,\tau)$ is said to be order bounded if it is contained in some order interval. \end{definition}


\begin{definition} \cite{Das} A sequence $\{x_n\}$ in a topological space $(X, \tau)$ is said to be $\mathcal{I}$-convergent to $x\in X$ if for any $U\in \tau$ containing $x,$ $\{n\in \mathbb{N}:~ x_n\notin U\}\in \mathcal{I}.$
\end{definition}

Naturally a prominent question may arise that, does there exists any notion of convergence in a topological vector space which could transform the core factor \emph{\textbf{degree of roughness}} ~`$r$' to a \emph{\textbf{neighborhood}} ~`$V$' of a topological vector space $(X, \tau).$ For answering this question, a different aspect of rough weighted $\mathcal{I}$-convergence been prominently discussed, following the concepts of Definition 1.1 and 1.2, specifically by replacing ~`$\theta$-metric space $X$', ~`$r$' and ~`$\varepsilon$' by the ~`locally solid Riesz space $L$', ~`$V$' (where $V$ is a fixed $\tau$-neighborhood of the zero element of $L$) and ~`$U$' (where $U$ is any arbitrary $\tau$-neighborhood of the zero element of $L$) respectively. We introduce the following definitions:

 \begin{definition} Let $\mathcal{I}$ be an admissible ideal of $\mathbb{N}$ and $t=\{t_{n}\}_{n\in \mathbb{N}}$ be a weighted sequence of real numbers. A sequence $x=\{x_{n}\}$ in a locally solid Riesz space $(L,\tau)$ is said to be rough weighted $\mathcal{I}_\tau$-convergent to $x_*\in L$ w.r.t the roughness of degree $V$ (where $V$ is a $\tau$-neighborhood of $\theta$) if for every $\tau$-neighborhood $U$ of $\theta,$ denoted by $x_{n}\xrightarrow[V]{W\mathcal{I}} x_*,$  the following expression holds,
  $$\{n\in \mathbb{N}: t_n(x_n-x_*)\notin V+U\}\in \mathcal{I}.$$
 We shall write $W\mathcal{I}_{\tau}-LIM^Vx=\{x_*\in X: x_{n}\xrightarrow[V]{W\mathcal{I}} x_*\}$ to denote the set of all rough weighted $\mathcal{I}$-limit points of the sequences $x=\{x_{n}\}$ with degree of roughness $V.$
  \end{definition}

\begin{definition} An element $c\in L$ is called weighted $\mathcal{I}_{\tau}$-cluster point of a sequence $x=\{x_n\}$ in $L$ for every $\tau$-neighborhood $U$ of $\theta,$ denoted by $x_{n}\xrightarrow{W_{\mathcal{I}}\Gamma_x^\tau} c,$ the set
$$\{n\in \mathbb{N}: t_n(x_n-c)\in U\}\notin \mathcal{I}.$$
The set of all weighted $\mathcal{I}_\tau$-cluster points of the sequence $x=\{x_n\}_{n\in \mathbb{N}}$ is denoted by $W_{\mathcal{I}}\Gamma_x^\tau.$
\end{definition}

Our main objective is to interpret the topological structure of the new convergence and characterize the rough weighted $\mathcal{I}_\tau$-limit set and weighted $\mathcal{I}_\tau$-cluster points set in a locally solid Riesz space. In addition, we give some results about the relationship between the sets $W\mathcal{I}_\tau-LIM^V x$ and  $W_\mathcal{I}\Gamma^\tau_x.$\\

\section{{Main Results }}

Followed by the Definition 1.8 of rough weighted $\mathcal{I}_\tau$-convergence over locally solid Riesz spaces we give the necessary condition for $W\mathcal{I}_{\tau}-LIM^Vx$ to be convex.

\begin{theorem} If $V$ is convex then the set $W\mathcal{I}_{\tau}-LIM^Vx$ is convex. \end{theorem}

\begin{proof} Let $U$ be an arbitrary $\tau$-neighborhood of zero. Then there exists a $W\in \mathcal{N}_{sol}$ such that $W+W\subseteq U.$ Since $W$ is a balanced set, we get $\lambda W\subseteq W$ and $(1-\lambda) W\subseteq W$ for all $0<\lambda <1.$ If $x_*,y_*\in W\mathcal{I}_\tau-LIM^{V}x$ then we have $A=\{k\in \mathbb{N}: t_k(x_k-x_*)\in V+W\}\in \mathcal{F}(\mathcal{I}),$ $B=\{k\in \mathbb{N}: t_k(x_k-y_*)\in V+W\}\in \mathcal{F(\mathcal{I})},$ where $\mathcal{F(\mathcal{I})}=\{M\subseteq \mathbb{N}: \mathbb{N}\setminus M \in \mathcal{I}\}$ is a filter associate with the ideal $\mathcal{I}.$ Consequently $A\cap B\in \mathcal{F(\mathcal{I})}.$  Further let $k\in A\cap B.$ Thus
  \begin{multline*}t_k[x_k-\{\lambda x_*+(1-\lambda)y_*\}]=\lambda t_k(x_k-x_*)+(1-\lambda) t_k(x_k-y_*)\\
  \in \lambda (V+W)+ (1-\lambda)(V+W)\subseteq V+W+W \subseteq V+U.\end{multline*}

  This implies $A\cap B \subseteq \{k\in \mathbb{N}: t_k[x_k-\{\lambda x_*+(1-\lambda)y_*\}]\in V+U\}$. So it follows that $W\mathcal{I}_{\tau}-LIM^Vx$ is convex.
\end{proof}

The converse of Theorem 2.1 is not true in general. To prove this important fact, we consider an example as follows:


 \begin{example}
 Let $\mathcal{I}$ is the ideal of subsets of $\mathbb{N}$ of natural density zero. Let us consider the locally solid Riesz Space $(\mathbb{R}^2,||\cdot||)$ with the max norm $||\cdot||$ and coordinate-wise ordering. The family $\mathcal{N}_{sol}$ of all $U(\varepsilon)$ defined as $U(\varepsilon)=\{\alpha\in \mathbb{R}^2 :||\alpha||<\varepsilon\}$ where $\varepsilon>0,$ constitutes a base at $\theta=(0,0).$ Let us define the sequence $x=\{x_{n}\}_{n\in \mathbb{N}},$ the weighted sequence $t=\{t_{n}\}_{n\in \mathbb{N}}$ and $\tau$-neighborhood of $\theta,$  say $V,$ in the following manner:
    \begin{equation*}
     x_{n}=\begin{cases} (1,0)&\mbox {if} ~ n\neq m^{2}~ \mbox{for all}~ m\in \mathbb{N}, \\
          (n,n)&\mbox{otherwise},
          \end{cases}\end{equation*}
          \[t_{n}= n~\mbox{ for all}~ n\in \mathbb{N},~\mbox{and}~ V=\{(\xi,\eta)\in \mathbb{R}^2: \parallel (\xi,\eta)\parallel\leq 1\}\cup \{(\xi,\eta)\in \mathbb{R}^2: \xi=\eta\}.\] Therefore $ W\mathcal{I}_\tau-LIM^{V}x= \{(\xi,\eta)\in \mathbb{R}^2: \xi-\eta=1\}$ is a convex set but not $V.$ $\Box$ \end{example}

Eliminating the condition of convexity from the aforementioned theorem, then the result may not be true.


\begin{example} Consider the locally solid Riesz Space $(\mathbb{R}^2,||\cdot||),$ where $||(\xi, \eta)||=|\xi|+|\eta|$ for $(\xi, \eta)\in \mathbb{R}^2$ and coordinate-wise ordering. The family $\mathcal{N}_{sol}$ of all $U(\varepsilon)$ defined as $U(\varepsilon)=\{\alpha\in \mathbb{R}^2 :||\alpha||<\varepsilon\}$ where $\varepsilon>0,$ constitutes a base at $\theta=(0,0).$ Consider the similar sequence as in Example 1, weighted sequence $t_{n}=e$ for all $n\in \mathbb{N}$ and $V=\{(\xi,\eta)\in \mathbb{R}^2: ||(\xi, \eta)||\leq \pi\}\cup \{(\xi,\eta)\in \mathbb{R}^2: \eta=\sin\xi\}.$ Taking the similar ideal as in example 1, we get $W\mathcal{I}_\tau-LIM^{V}x= \{(\xi,\eta)\in \mathbb{R}^2: |1-\xi|+|\eta|\leq\pi\}\cup \{(\xi,\eta)\in \mathbb{R}^2: e\eta=\sin (e\xi-e)\},$ which is not a convex set. So we conclude that if $V$ is not a convex set then the set $W\mathcal{I}_{\tau}-LIM^Vx$ may or may not be convex.
\end{example}

Ghosal et al. \cite[Theorem 3.1]{Ghosal2020}, had shown that if the weighted sequence is not $\mathcal{I}$-bounded the the set $W\mathcal{I}-LIM^{r}x$ contains at most one element. While reformulating the above theorem based over locally solid Riesz Space, the object of~``$W\mathcal{I}-LIM^{r}x$ contains at most one element" violates. We exemplify this assertion below i.e., if the weighted sequence is not $\mathcal{I}$-bounded then the set $W\mathcal{I}_{\tau}-LIM^Vx$ may not be singleton infact it may be infinite and $\tau$-unbounded.


\begin{example} Consider the ideal $\mathcal{I}=\{A\subset \mathbb{N}: \displaystyle \sum_{a\in A} a^{-1}<\infty\}$ and the locally solid Riesz Space $(\mathbb{R}^2,||\cdot||),$ where $||\cdot||$ is the Euclidean norm. Let us define the sequence $x=\{x_{n}\}_{n\in \mathbb{N}}$ and the weighted sequence $t=\{t_{n}\}_{n\in \mathbb{N}}$ in the following manner:\\
    \begin{equation*}
     x_{n}=\begin{cases} ((-1)^{n},0)  &\mbox {if} ~ n\neq m^{p}~ \mbox{for all}~ m\in \mathbb{N} ~\mbox{and}~ p\in \mathbb{N}\setminus\{1\}, \\
          (n,0)&\mbox{otherwise},
          \end{cases}
          ~\mbox{and}~
     t_{n}= n~ \mbox{for all}~ n\in \mathbb{N}.\\
          \end{equation*}
   It is very obvious that the sequence $x=\{x_n\}_{n\in \mathbb{N}}$ is not weighted $\mathcal{I}_\tau$-convergent to any point of $\mathbb{R}^2.$
   Let $V=\{(\xi,\eta)\in \mathbb{R}^2 :-1<\eta<1\}.$ Then $W\mathcal{I}_\tau-LIM^{V}x=\{(\xi,\eta)\in \mathbb{R}^2 :\eta=0\}.$  This example shows that the sequence $\{x_{n}\}_{n\in \mathbb{N}}$ is not weighted $\mathcal{I}_\tau$-convergent to any point but the rough weighted $\mathcal{I}_\tau$-limit set is $\{(\xi,\eta)\in \mathbb{R}^2 :\eta=0\}$ which is infinite and $\tau$-unbounded. $\Box$ \end{example}


\begin{theorem} If the weighted sequence $t=\{t_n\}_{n\in \mathbb{N}}$ is not $\mathcal{I}$-bounded and $V$ be $\tau$-bounded then the rough weighted $\mathcal{I}_\tau$-limit set $W\mathcal{I}_\tau-LIM^{V}x$ of a sequence $x=\{x_n\}_{n\in \mathbb{N}}$ in locally solid Hausdorff Riesz spaces can have at most one element. \end{theorem}

\begin{proof} If $W\mathcal{I}_\tau-LIM^{V}x= \varnothing,$ the theorem is obvious. So assuming  $W\mathcal{I}_\tau-LIM^{V}x\neq \varnothing.$ If possible let $x_*$ and $y_*$ be two distinct elements in $W\mathcal{I}_\tau-LIM^{V}x.$ Let $\alpha=x_*-y_*.$ If $\alpha\neq \theta$ then there exists $U\in \mathcal{N}_{sol}$ such that $\theta \in U$ but $\alpha \notin U$ (since the Riesz space is $T_2$).

  Now $V$ is $\tau$-bounded so is $B=V+V+V.$ So there exists a positive real number $p$ such that $B\subset pU.$ Let $A_1=\{k\in \mathbb{N}: t_k\geq p\}.$ Since $t=\{t_n\}_{n\in \mathbb{N}}$ is not $\mathcal{I}$-bounded so  $A_1\notin \mathcal{I}.$ Since $V$ is a $\tau$-neighborhood of $\theta$ so there exists $W\in \mathcal{N}_{sol}$ such that $W+W\subset V.$ Let
 $$A_2=\{k\in \mathbb{N}: t_k(x_k-x_*)\in V+W\}~\mbox{and}~A_3=\{k\in \mathbb{N}: t_k(x_k-y_*)\in V+W\},$$
 then $A_2, A_3\in \mathcal{F(\mathcal{I})}.$ Let $A=A_1 \cap A_2\cap A_3$ then $A\neq \varnothing$ as well as $A$ is an infinite subset of $\mathbb{N}.$ For $k\in A,$
$$t_k(x_*-y_*)=t_k(x_*-x_k)+t_k(x_k-y_*)\in V+W+V+W= V+V+W+W\subseteq B.$$
Therefore $A\subseteq \{k\in \mathbb{N}: t_k\alpha\in B\}$ and so $\{k\in \mathbb{N}: t_k\alpha\in B\}\neq \varnothing,$ i.e., $t_k\alpha\in B$ for all $k\in A.$\\

Since $\alpha \notin U,$ we get $t_k\alpha\notin t_kU$ for all $k\in A.$
$$\{t_k\alpha:k\in A\}\nsubseteq t_lU ~\Rightarrow~ B\subseteq t_lU ~\mbox{for no}~l\in A,$$
which is a contradiction. Hence $\alpha=\theta$ and our result is established.
\end{proof}

Our next example proves that if $V$ be $\tau$-bounded then the set $W\mathcal{I}_\tau-LIM^{V}x$ may not be an order bounded set.


 \begin{example} Let $\mathcal{I}$ is the ideal of subsets of $\mathbb{N}$ of natural density zero and $L$ be the space of all Lebesgue measurable functions on $I=[0,1]$ with the usual point-wise ordering, i.e., for $x,y\in L,$ we define $x\leq y$ if and only if $x(t)\leq y(t)$ for every $t\in I.$ Consider the map $||.||:L\rightarrow \mathbb{R}$ defined by $||x||=( {\int_{I}}  x^2(s)ds)^\frac{1}{2},$ where $x\in L.$ Then $(L,\tau)$ forms a locally solid Riesz space. Put $V=\{x\in L :||x||\leq 1\}.$  Again we consider the weighted sequence $t=\{t_{n}\}_{n\in \mathbb{N}}$ and  the sequence $x=\{x_{n}\}_{n\in \mathbb{N}}$ in the following manner:
                  \begin{equation*}
     t_{n}= \begin{cases} 2+\frac{1}{n}  &\mbox {if} ~ n\neq m^{2}~ \mbox{for all}~ m\in \mathbb{N}, \\
          3  &\mbox{otherwise},
          \end{cases}
          ~\mbox{and}~
     x_{n}(s)=\begin{cases} \frac{1}{n}  &\mbox {if} ~ n\neq m^{2}~ \mbox{for all}~ m\in \mathbb{N}, \\
          1 &\mbox{otherwise},
          \end{cases}
          \end{equation*}
 for all $s\in I,$ $n\in \mathbb{N}.$   Thus it follows that $W\mathcal{I}_\tau-LIM^{V}x=\frac{1}{2}V$ is $\tau$-bounded but not order bounded. $\Box$ \end{example}


\begin{theorem} Let $(L,\tau)$ be a Hausdorff locally solid Riesz space and $V\in \mathcal{N}_{sol}$ such that $V+V+U\in \mathcal{N}_{sol}$ for all $U\in \mathcal{N}_{sol}.$ For a sequence $x=\{x_{n}\}_{n\in \mathbb{N}}$ in $L,$ we have  $$x_*-y_*\in \frac{\overline{V+V}}{\displaystyle{\inf_{n\in \mathbb{N}}} t_n}  ~\mbox{for all}~ x_*,y_*\in W\mathcal{I}_\tau-LIM^{V}x.$$
In addition, if $V$ is $\tau$-bounded then $W\mathcal{I}_\tau-LIM^{V}x$ is a $\tau$-bounded set. \end{theorem}

\begin{proof} Let $U$ be an arbitrary $\tau$-neighborhood of zero. Then there exist  $U_0, U_1\in \mathcal{N}_{sol}$ such that $U_0\subseteq U$ and $U_1+U_1\subseteq U_0.$ Thus $K_1,K_2\in \mathcal{F}(\mathcal{I}),$ where $K_1=\{k\in \mathbb{N}: t_k(x_k-x_*)\in V+U_1\}$ and $K_2=\{k\in \mathbb{N}: t_k(x_k-y_*)\in V+U_1\}.$ Consequently $K_1\cap K_2\in \mathcal{F}(\mathcal{I}).$

  Moreover if $k\in K_1\cap K_2$ then $t_k(x_*-y_*)=t_k(x_*-x_k)+t_k(x_k-y_*)\in V+U_1+V+U_1\subseteq V+V+U_0. $  Thus we get  $t_k(x_*-y_*)\in V+V+U_0$ for all $k\in K_1\cap K_2.$ Also for $k\in K_1\cap K_2,$
  $$\displaystyle{\inf_{n\in \mathbb{N}}} t_n |(x_*-y_*)| \leq t_k |(x_*-y_*)|\in V+V+U_0.$$
  This implies $(x_*-y_*) \in \frac{V+V}{\displaystyle{\inf_{n\in \mathbb{N}}} t_n}+ \frac{U_0}{\displaystyle{\inf_{n\in \mathbb{N}}} t_n}\subseteq
  \frac{V+V}{\displaystyle{\inf_{n\in \mathbb{N}}} t_n}+ \frac{U}{\displaystyle{\inf_{n\in \mathbb{N}}} t_n}.$ Hence $(x_*-y_*) \in \frac{\overline{V+V}}{\displaystyle{\inf_{n\in \mathbb{N}}} t_n}$ since $(L,\tau)$ is Hausdorff and the intersection of all $\tau$-neighborhoods $U$ of zero is the singleton $\{\theta\}.$\\

  Now we proceed to the second part of the theorem. Let $W\mathcal{I}_\tau-LIM^{V}x\neq \varnothing$ and $x_*\in W\mathcal{I}_\tau-LIM^{V}x.$ From the above argument we get
  $W\mathcal{I}_\tau-LIM^{V}x\subseteq \frac{\overline{V+V}}{\displaystyle{\inf_{n\in \mathbb{N}}} t_n}+\{x_*\}.$  \end{proof}

\begin{remark} Naturally a question may arise that, does there exist any neighbourhood $V\in \mathcal{N}_{sol}$ in a Hausdorff locally solid Riesz space such that $V+V+U\in \mathcal{N}_{sol}$ for all $U\in \mathcal{N}_{sol}$ ?

For answering the above question we consider the Hausdorff locally solid Riesz space same as in Example 1 and define the neighbourhood $V\in \mathcal{N}_{sol}$ in the following manner: $V=U(1)\in \mathcal{N}_{sol}$ where $U(\varepsilon)=\{\alpha\in \mathbb{R}^2 :||\alpha||<\varepsilon\}$ for all $\varepsilon>0$ and it is clear that $V+V+U(\varepsilon)\in \mathcal{N}_{sol}$ for all $U(\varepsilon)\in \mathcal{N}_{sol}.$ \end{remark}

Further we discuss the closeness of rough weighted $\mathcal{I}_\tau$-limit set over locally solid Riesz spaces. Our next example assures that rough weighted $\mathcal{I}_\tau$-limit set is not closed.


 \begin{example} Consider the locally solid Riesz Space $\mathbb{R}^2$ and the ideal $\mathcal{I}$ as in Example 3. Let us define the sequence $x=\{x_{n}\}_{n\in \mathbb{N}}$ and the weighted sequence $t=\{t_{n}\}_{n\in \mathbb{N}}$ in the following manner;\\
    \begin{equation*}
     x_{n}=\begin{cases} (1,-\frac{2}{n})  &\mbox {if} ~ n\neq m^{2}~ \mbox{for all}~ m\in \mathbb{N}, \\
          (n,-n)&\mbox{otherwise}
          \end{cases}
          ~\mbox{and} ~ t_{n}= n~ \mbox{for all}~ n\in \mathbb{N}.
          \end{equation*}

    Let $V=\{(\xi,\eta)\in \mathbb{R}^2 :-1<\xi<1, \eta\geq -1\}.$ Therefore $W\mathcal{I}_\tau-LIM^{V}x=\{(\xi,\eta)\in \mathbb{R}^2 :\xi=1,\eta<0\}$ is not closed. In fact it is not an open set. $\Box$ \end{example}

A careful inspection of the previous examples exhibit how the set $W\mathcal{I}_\tau-LIM^{V}x$ could be generalized. We do so in the next theorem.


 \begin{theorem} Let $(L,\tau)$ be a Hausdorff locally solid Riesz space. Hence we assert
 \begin{equation*}
     W\mathcal{I}_\tau-LIM^{V}x \begin{cases} \mbox{is closed}; ~\mbox{if weighted sequence be $\mathcal{I}$-bounded}, \\
          \mbox{is closed};~\mbox{if weighted sequence is not $\mathcal{I}$-bounded and} ~V ~\mbox{be} ~\tau\mbox{-bounded},\\
          \mbox{has no definite conclusion};~\mbox{if weighted sequence is not $\mathcal{I}$-bounded and}\\
             V ~\mbox{be} ~\tau\mbox{-unbounded}.\\
             \end{cases}
          \end{equation*}
          \end{theorem}

\begin{proof} \emph{\textbf{Case 1:}} Let the weighted sequence $t=\{t_n\}_{n\in \mathbb{N}}$ be $\mathcal{I}$-bounded. In this case there exists a positive real number $M$ such that $K_1=\{k\in \mathbb{N}: t_k<M\}\in \mathcal{F}(\mathcal{I}).$ Assume $p_*\in \overline{W\mathcal{I}_\tau-LIM^{V}x}.$ Then there exists a sequence  $\{p_n\}_{n\in \mathbb{N}}$ in $W\mathcal{I}_\tau-LIM^{V}x$ such that $p_n\rightarrow p_*$ as $n\rightarrow \infty$ (by using 1st countable property of L). Naturally  $M p_n\rightarrow M p_*$ as $n\rightarrow \infty.$

Again we consider $U$ is any neighborhood of $\theta$ and corresponding to $U$ there exists a $U_0\in \mathcal{N}_{sol}$ such that $U_0+U_0\subset U.$ Hence  $M(p_n-p_*)\in U_0$ for all $n\geq k_0,$ where $k_0$ is a positive integer depends on $U_0.$

On the other hand if $k\in K_1$ then $t_k<M$ implies $t_k|p_{k_0}-p_*|\leq M|p_{k_0}-p_*|$ and finally $|t_k(p_{k_0}-p_*)|\leq |M(p_{k_0}-p_*)|.$

Since $M(p_{k_0}-p_*)\in U_0$ and $U_0$ is a solid set so $t_k(p_{k_0}-p_*)\in U_0$ for all $k\in K_1.$\\

We define the set $K_2=\{k\in \mathbb{N}: t_k(x_k-p_{k_0})\in V+U_0\}.$  Then $p_{k_0}\in W\mathcal{I}_\tau-LIM^{V}x$ follows $K_2\in \mathcal{F}(\mathcal{I}).$
For $k\in K_1\cap K_2,$ $t_k(x_k-p_*)=t_k(x_k-p_{k_0})+t_k(p_{k_0}-p_*)\in V+U_0+U_0\subset V+U.$

This implies $K_1\cap K_2 \subset \{k\in \mathbb{N}: t_k(x_{k}-p_*)\in V+U\}.$ Clearly $\{k\in \mathbb{N}: t_k(x_{k}-p_*)\in V+U\}\in \mathcal{F}(\mathcal{I}).$ So we conclude case 1.\\

\emph{\textbf{Case 2:}} If the weighted sequence $t=\{t_n\}_{n\in \mathbb{N}}$ is not $\mathcal{I}$-bounded and $V$ is $\tau$-bounded then from Theorem 2.2, the set $W\mathcal{I}_\tau-LIM^{V}x$ becomes either singleton or empty. Hence it is closed.\\

\emph{\textbf{Case 3:}} From Examples 1, 2 and 5 it is clear that the the set $W\mathcal{I}_\tau-LIM^{V}x$ neither open nor closed.\end{proof}

We initiate this section with the definitions of weighted $\tau$-boundedness and weighted $\mathcal{I}\tau$-boundedness over a locally solid Riesz space $L.$


\begin{definition} Let $(L,\tau)$ be a locally solid Riesz space and $t=\{t_{n}\}_{n\in \mathbb{N}}$ is a weighted sequence. A sequence $x=\{x_{n}\}_{n\in \mathbb{N}}$ in $L$ is said to be weighted $\tau$-bounded if for every $\tau$-neighborhood $U$ of $\theta$ there exists some $\lambda>0$ such that  $\lambda t_kx_k\notin U$  at most for finitely many $k.$ \end{definition}


 \begin{definition} Let $(L,\tau)$ be a locally solid Riesz space and $t=\{t_{n}\}_{n\in \mathbb{N}}$ is a weighted sequence. A sequence $x=\{x_{n}\}_{n\in \mathbb{N}}$ in $L$ is said to be weighted $\mathcal{I}_\tau$-bounded if for every $\tau$-neighborhood $U$ of $\theta$ there exists some $\lambda>0$ such that  $\{k\in \mathbb{N}: \lambda t_kx_k\notin U\}\in\mathcal{I}.$ \end{definition}


\begin{theorem} If a sequence $x=\{x_{n}\}_{n\in \mathbb{N}}$ be weighted $\mathcal{I}_{\tau}$-bounded then for every $\tau$-neighborhood $V$ of $\theta$ there exists a positive real number $\mu$ such that $W\mathcal{I}_{\tau}-LIM^{\mu V}x\neq \varnothing.$ \end{theorem}

\begin{proof} Since $x=\{x_{n}\}_{n\in \mathbb{N}}$ is weighted $\mathcal{I}_\tau$-bounded then for every $\tau$-neighborhood $V$ of $\theta$ there exists some $\lambda>0$ such that  $$\{k\in \mathbb{N}: \lambda t_kx_k\in V\}=\{k\in \mathbb{N}: \lambda t_k(x_k-\theta)\in V\}\in \mathcal{F}(\mathcal{I}).$$

Therefore $\{k\in \mathbb{N}: t_k(x_k-\theta)\in \frac{1}{\lambda}V+U\}\in \mathcal{F}(\mathcal{I}),$ where $U$ is an arbitrary $\tau$-neighborhood of $\theta.$
 Setting $\mu=\frac{1}{\lambda}.$ Hence the set $W\mathcal{I}_{\tau}-LIM^{\mu V}x$ contains the null element of $L.$ \end{proof}

 But the converse is not true. We choose an example to emphasis our assertion.



 \begin{example} Consider the ideal $\mathcal{I}$ and the locally solid Riesz Space $\mathbb{R}^2$ as defined in Example 3.  Let us define the sequence $x=\{x_{n}\}_{n\in \mathbb{N}}$ and the weighted sequence $t=\{t_{n}\}_{n\in \mathbb{N}}$ in the respective order:
    \begin{equation*}
     x_{n}=\begin{cases} (2+\frac{1}{n},1+\frac{3}{n})  &\mbox {if} ~ n\neq m^{2}~ \mbox{for all}~ m\in \mathbb{N}, \\
          (5,-5) &\mbox{otherwise}
          \end{cases}
         ~\mbox{and}~~
     t_{n}= \sqrt{n}~ \mbox{for all}~ n\in \mathbb{N}.
          \end{equation*}

   If $V$ be any arbitrary $\tau$-neighborhood of $\theta$ and $\mu$ be any positive real number then obviously $(2,1)\in W\mathcal{I}_\tau-LIM^{\mu V}x$ and so $W\mathcal{I}_\tau-LIM^{\mu V}x\neq \varnothing.$  For any positive real number $\lambda$ we get
     \begin{equation*}
     \lambda t_n x_{n}=\begin{cases} (2\lambda\sqrt{n}+\frac{1}{\sqrt{n}},\lambda \sqrt{n}+\frac{3}{\sqrt{n}})  &\mbox {if} ~ n\neq m^{2}~ \mbox{for all}~ m\in \mathbb{N}, \\
          (5\lambda \sqrt{n},-5\lambda \sqrt{n})&\mbox{otherwise}.
          \end{cases}
          \end{equation*}

 In this case $\{k\in \mathbb{N}:\lambda t_kx_k\notin V\}\notin \mathcal{I}$ Therefore the sequence $x=\{x_{n}\}_{n\in \mathbb{N}}$ is not weighted $\mathcal{I}_\tau$-bounded although $W\mathcal{I}_\tau-LIM^{\mu V}x\neq \varnothing.$  $\Box$ \end{example}

If we consider the sequence $x_{n}=(\frac{1}{n},0)$ for all $n\in \mathbb{N}$ in $\mathbb{R}^2$ as taken in Example 3, $t_n=n^2~\mbox{for all}~ n\in \mathbb{N}$ and $\mathcal{I}$ an arbitrary ideal. Hence the sequence $x=\{x_{n}\}_{n\in \mathbb{N}}$ is $\mathcal{I}_\tau$-bounded as well as $\tau$-bounded, but $W_{\mathcal{I}}\Gamma_x^\tau=\varnothing.$ \\

As an immediate consequence the following example shows that the set $W_{\mathcal{I}}\Gamma_x^\tau$ is empty even if the sequence $x=\{x_{n}\}_{n\in \mathbb{N}}$ is weighted $\mathcal{I}_\tau$-bounded.


 \begin{example} Consider the infinite dimensional normed space $l^2$ space with the norm
 $$||\alpha||=\left(\displaystyle{\sum_{j=1}^{\infty}}~|\xi_j|^2 \right)^{\frac{1}{2}}~\mbox{where}~\alpha=(\xi_1,\xi_2,...)\in l^2$$ with coordinatewise ordering. So $l^2$ is a locally solid Riesz Space. The family $\mathcal{N}_{sol}$ of all $U(\varepsilon)$ defined as $U(\varepsilon)=\{\alpha\in l^2 :||\alpha||<\varepsilon\}$ where $\varepsilon>0,$ constitutes a base at $\theta=(0,0,...).$ We consider the ideal $\mathcal{I}=\{K\subset \mathbb{N}:\displaystyle{\lim_{n\rightarrow \infty}}\frac{|K\cap \{1,2,3,...,n\}|}{n}=0\},$ the sequence  $\{x_{n}\}_{n\in \mathbb{N}}$ in $l^2$ such that
 $x_{n}=e_n, ~\mbox{where}~e_{n}~\mbox{has}~n^{{th}}~\mbox{term}~ 1~ \mbox{and other terms are}~ 0$
 and $t_n=2~\mbox{for all}~ n\in \mathbb{N}.$ Therefore the sequence $x=\{x_{n}\}_{n\in \mathbb{N}}$ is weighted $\mathcal{I}_\tau$-bounded as well as $\mathcal{I}_\tau$-bounded but $W_{\mathcal{I}}\Gamma_x^\tau=\varnothing.$ $\Box$ \end{example}

In the above example, if we reassume $t_{n}= n$ for all $n\in \mathbb{N}$ keeping the space $l^2$ and the sequence $x=\{x_{n}\}_{n\in \mathbb{N}}$ unaltered, the sequence $x=\{x_{n}\}_{n\in \mathbb{N}}$ is $\mathcal{I}_\tau$-bounded as well as  $\tau$-bounded but $W_{\mathcal{I}}\Gamma_x^\tau=\varnothing.$\\

On the other hand, another important question may arise that  does the set $W_{\mathcal{I}}\Gamma_x^\tau$ is compact if the space is infinite dimensional? Answer is no. We sketch an important example below to answer this question.


 \begin{example}
 Let $\mathbb{N}= \displaystyle{\bigcup_{j=1}^\infty} \Delta_j$ where $\Delta_j=\{2^{j-1}(2s-1):s\in \mathbb{N}\}$ and $\mathcal{I}=\{A\subset \mathbb{N}: A\cap \Delta_j\neq \varnothing ~\mbox{for finitely many}~ j\}.$ Then $\mathcal{I}$ forms an admissible ideal. Consider another decomposition of $\mathbb{N},$ i.e.,  $D_r=\{p^{s}_r:s\in \mathbb{N}\},$ for all $r\in \mathbb{N}\setminus\{1\}, \{p_2<p_3<p_4,...\}$ is a sequence of distinct primes and $D_1=\mathbb{N}\setminus \displaystyle{\bigcup_{r=2}^\infty} D_r.$ Setting $t_n=n,$ if $n\in \mathbb{N}$ and the sequence $x=\{x_{n}\}_{n\in \mathbb{N}}$ such that $x_{n}=e_r, ~\mbox{for all}~n\in D_r$ (in the locally solid Riesz space $l^2$ as above).
 For each $0<\varepsilon<1$ and $r\in \mathbb{N},$ $\{k\in \mathbb{N}: t_{k}(x_{k}-e_r)\in U(\varepsilon)\}\notin \mathcal{I}.$ This shows that $e_r\in W_\mathcal{I}\Gamma_x^\tau$ for all $r\in \mathbb{N}.$  Let $A=\{e_1,e_2,e_3,...\}.$ Then, $A(\subset W_\mathcal{I}\Gamma_x^\tau)$ is closed but not compact. So the set $W_\mathcal{I}\Gamma_x^\tau$ is not compact.  $\Box$ \end{example}

   In \cite{F2} Fridy shown that for any number sequence $x=\{x_{n}\}_{n\in \mathbb{N}},$ the statistical cluster points set of $x$ is closed.
 Also in \cite{Das} Das had shown that in a topological space the $\mathcal{I}$-cluster points set is closed. But the following example shows that in general, the weighted $\mathcal{I}_\tau$-cluster points set in a locally solid Riesz space may not be closed.


 \begin{example}  Consider two decompositions of $\mathbb{N},$ as in previous Example 8 and the locally solid Riesz Space $\mathbb{R}^2$ same as Example 3.  Let us define the sequence $x=\{x_{n}\}_{n\in \mathbb{N}}$ and the weighted sequence $t=\{t_{n}\}_{n\in \mathbb{N}}$  in the following manner,
        $x_{k}=(\frac{1}{j}+\frac{1}{k^2},0)~\mbox{for all}~ ~k\in D_j~(\mbox{where}~j=1,2,3,...)$ and $t_{k}= k~ \mbox{for all}~ k\in \mathbb{N}.$

             Let $U$ be a $\tau$-neighborhood of $\theta,$ so there exists some $U(\varepsilon)\in \mathcal{N}_{sol},$ $\varepsilon>0$ such that $U(\varepsilon)\subset U.$ Then for each $j\in \mathbb{N},$ we get
  $\{k\in \mathbb{N}: {t_k}(x_k-(\frac{1}{j},0))\in U(\varepsilon)\}\notin \mathcal{I}.$  This shows that $(\frac{1}{j},0)\in W_\mathcal{I}\Gamma_x^\tau$ for all $j\in \mathbb{N}.$ Next we assume $k\in \mathbb{N}$ then there exists an integer $j\in \mathbb{N}$ such that $k\in D_j$ for some $j\in \mathbb{N}.$ If $k\in D_j$ for some $j\in \mathbb{N}\setminus\{1\},$ then $k$ is of the form $k=p^{s}_j$ where $s\in \mathbb{N}$ then
$$t_k(x_k-\theta)=p^{s}_j\left(\frac{1}{j}+\frac{1}{p^{2s}_j},0\right)=\left(\frac{p^{s}_j}{j}+\frac{1}{p^{s}_j},0\right)\geq (1,0).$$
Then  $\{k\in \mathbb{N}: {t_k}(x_k-\theta)\in U(\alpha)\}\in \mathcal{I},$ where $0<\alpha<1.$ This implies $\theta\notin W_\mathcal{I}\Gamma_x^\tau.$ On the other hand if $k\in D_1,$ then $\{k\in \mathbb{N}: {t_k}(x_k-\theta)\in U(1)\}\in \mathcal{I}.$ Therefore the set $W_\mathcal{I}\Gamma_x^\tau$ is not closed. \end{example}


\begin{theorem} For a sequence $x=\{x_{n}\}_{n\in \mathbb{N}},$ the weighted~ $\mathcal{I_{\tau}}$-cluster points set $W_\mathcal{I}\Gamma_x^\tau$ is closed if the weighted sequence $\{t_{n}\}_{n\in \mathbb{N}}$ is $\mathcal{I}$-bounded.
\end{theorem}

\begin{proof} As the weighted sequence $\{t_{n}\}_{n\in \mathbb{N}}$ is $\mathcal{I}$-bounded, there exists a non negative real number $M$ such that the set $A=\{n\in \mathbb{N}:t_n< M\}\in \mathcal{F(\mathcal{I})}.$
 Let $U$ is any neighborhood of $\theta$ and so there exists a $U_0\in \mathcal{N}_{sol}$ such that $U_0+U_0\subset U.$ Consider a sequence $\{p_n\}_{n\in \mathbb{N}}\in W_\mathcal{I}\Gamma_x^\tau$ such that $p_n\rightarrow p.$ Hence  $M(p_n-p)\in U_0$ for all $n\geq k_0,$ where $k_0$ is a positive integer depends on $U_0.$

 Also $p_{k_0}\in W_\mathcal{I}\Gamma_x^\tau$ implies $B=\{n\in \mathbb{N}:t_n(x_n-p_{k_0})\in U_0\}\notin \mathcal{I}.$ Therefore $A\cap B\notin \mathcal{I}$ otherwise $A^c\cup (A\cap B)\in \mathcal{I}$ (since $A\in \mathcal{F(\mathcal{I})}$) we end up with a contradiction that $B\in \mathcal{I}.$ For $k\in A\cap B,$ we have $t_k(x_k-p)=t_k(x_k-p_{k_0})+t_k(p_{k_0}-p)\in U_0+U_0\subset U.$ Thus $A\cap B\subseteq \{n\in \mathbb{N}:t_n(x_n-p)\in U\}$ demonstrates that $p\in W_\mathcal{I}\Gamma_x^\tau$ and the closeness of $W_\mathcal{I}\Gamma_x^\tau$ is established.
\end{proof}


\begin{theorem}  Let $(L,\tau)$ be a locally solid Riesz space, $x_{n}\xrightarrow{W_\mathcal{I}\Gamma_x^\tau} c$ and $x_{n}\xrightarrow{W_\mathcal{I}\Gamma_x^\tau} d.$ Then \\
$(i)$ $|x_{n}|\xrightarrow{W_\mathcal{I}\Gamma_x^\tau} |c|,$ $(ii)$ $x^+_{n}\xrightarrow{W_\mathcal{I}\Gamma_x^\tau} c^+,$ $(iii)$ $x^-_{n}\xrightarrow{W_\mathcal{I}\Gamma_x^\tau} c^-,$\\
$(iv)$ $x_{n}\not\xrightarrow{W_\mathcal{I}\Gamma_x^\tau} c \vee d,$ $(v)$ $x_{n}\not\xrightarrow{W_\mathcal{I}\Gamma_x^\tau} c\wedge d.$ \end{theorem}

\begin{proof}  Proof of $(i), (ii)$ and $(iii)$ are similar to proof of Theorem 4.2 \cite{AP}, so omitted.\\
For the proof of $(iv)$ and $(v)$ we consider the locally solid Riesz Space $\mathbb{R}^2$ as in Example 3.  Let us define the sequence $x=\{x_{n}\}_{n\in \mathbb{N}}$ and the weighted sequence $t=\{t_{n}\}_{n\in \mathbb{N}}$ in the manner;\\
    \begin{equation*}
     x_{n}=\begin{cases} (1,-2)  ~\mbox {if} ~ n\in \{2^p: p\in \mathbb{N}\}, \\
          (-1,2)~~ \mbox{otherwise}
          \end{cases}
          ~\mbox{and}~ t_{n}= n~ \mbox{for all}~ n\in \mathbb{N} ~\mbox{respectively}.\\
          \end{equation*}

    Consider the ideal $\mathcal{I}$ as in Example 8. Then the weighted $\mathcal{I}_\tau$-cluster points set of the sequence $x,$
     $W_\mathcal{I}\Gamma_x^\tau=\{(1,-2), (-1,2)\}.$
    So $x_{n}\xrightarrow{W_\mathcal{I}\Gamma_x^\tau} (1,-2)$ and $x_{n}\xrightarrow{W_\mathcal{I}\Gamma_x^\tau} (-1,2)$ but $(1,2), (-1,-2)\notin W_\mathcal{I}\Gamma_x^\tau.$ Hence the results. \end{proof}

    From the above Theorem 2.6 it is clear that the set $W_\mathcal{I}\Gamma_x^\tau$ may not be a convex set.\\

    We finally draw a significant relationship between the sets $W\mathcal{I}_{\tau}-LIM^Vx$ and $W_\mathcal{I}\Gamma_x^\tau.$


\begin{theorem} Let $(L,\tau)$ be a Hausdorff locally solid Riesz space and $V\in \mathcal{N}_{sol}$ such that $V+U\in \mathcal{N}_{sol}$ for all $U\in \mathcal{N}_{sol}.$ If  $x=\{x_{n}\}_{n\in \mathbb{N}}$ be a sequence in $L,$ then
$$x_*-c\in \frac{\overline{V}}{\displaystyle{\inf_{n\in \mathbb{N}}} t_n} ~\mbox{where}~x_*\in W\mathcal{I}_{\tau}-LIM^Vx ~\mbox{and}~ c\in W_\mathcal{I}\Gamma_x^\tau.$$
Moreover if $V$ is $\tau$-bounded, $V+V+U\in \mathcal{N}_{sol}$ for all $U\in \mathcal{N}_{sol}$ and  $W\mathcal{I}_{\tau}-LIM^Vx$ is non-empty then $W_\mathcal{I} \Gamma_x^\tau$ is $\tau$-bounded. \end{theorem}

 \begin{proof} Let us consider $U$ be an arbitrary $\tau$-neighborhood of zero. Then there exist $U_0,U_1\in \mathcal{N}_{sol}$ such that $U_0\subseteq U$ and $U_1+U_1\subseteq U_0.$ We know that  $K_1\in \mathcal{F(\mathcal{I})}$ and $K_2\notin\mathcal{I},$ where $K_1=\{k\in \mathbb{N}: t_k(x_k-x_*)\in V+U_1\}$ and $K_2=\{k\in \mathbb{N}: t_k(x_k-c)\in U_1\}.$ So it follows that $K_1\cap K_2$ is nonempty and an infinite set.  Further $k\in K_1\cap K_2$ then $t_k(x_*-c)=t_k(x_*-x_k)+t_k(x_k-c)\in V+U_1+U_1\subseteq V+U_0. $  Thus we get  $t_k(x_*-c)\in V+U_0$ for all $k\in K_1\cap K_2.$ For $k\in K_1\cap K_2,$
  $$\displaystyle{\inf_{n\in \mathbb{N}}} t_n |(x_*-c)|\leq |\displaystyle{\inf_{k\in {K_1\cap K_2}}} t_k (x_*-c)| \leq  |t_k(x_*-c)|\in V+U_0.$$
  This implies $(x_*-c) \in \frac{V}{\displaystyle{\inf_{n\in \mathbb{N}}} t_n}+\frac{U_0}{\displaystyle{\inf_{n\in \mathbb{N}}} t_n}\subseteq \frac{V}{\displaystyle{\inf_{n\in \mathbb{N}}} t_n}+\frac{U}{\displaystyle{\inf_{n\in \mathbb{N}}} t_n}.$  Since $(L,\tau)$ is Hausdorff and the intersection of all $\tau$-neighborhoods $U$ of zero is the singleton set $\{\theta\}.$ This shows that $(x_*-c) \in \frac{\overline{V}}{\displaystyle{\inf_{n\in \mathbb{N}}} t_n}.$

 To start the second part of the theorem, we let $x_*\in W\mathcal{I}_{\tau}-LIM^Vx.$ Applying the above argument as well as Theorem 2.3, we get
  $W_\mathcal{I}\Gamma_x^\tau \left(\subseteq \frac{\overline{V}}{\displaystyle{\inf_{n\in \mathbb{N}}} t_n}+W\mathcal{I}_{\tau}-LIM^Vx \right)$ is $\tau$-bounded. Hence the results conclude.    \end{proof}

\begin{remark} In retrospect to Remark 1, if we consider $V=U(1)\in \mathcal{N}_{sol}$ where $U(\varepsilon)=\{\alpha\in \mathbb{R}^2 :||\alpha||<\varepsilon\}$ for all $\varepsilon>0$ prominently $V+U(\varepsilon), V+V+U(\varepsilon)\in \mathcal{N}_{sol}$ for all $U(\varepsilon)\in \mathcal{N}_{sol}.$ So the existence of such neighbourhood $V$ is guaranteed. \end{remark}

\begin{example} In view of the above Theorem 2.8 one would naturally like to seek an example for which $W_\mathcal{I}\Gamma_x^\tau$ will be $\tau$-unbounded, i.e., in general $W_\mathcal{I}\Gamma_x^\tau$ is not $\tau$-bounded. If we replace the sequence $x=\{x_{n}\}_{n\in \mathbb{N}}$ by the sequence $z=\{z_{n}\}_{n\in \mathbb{N}}$ where $z_n(t)=j$ for all $n\in \Delta_j$ (where $j=1,2,3,...$), for all $t\in I$ and $t_n=n$ for all $n\in \mathbb{N}$ as in Example 4 then we get $\{z_n:n\in \mathbb{N}\}\subset W_\mathcal{I}\Gamma_x^\tau$ where $\mathcal{I}$ is the ideal as in Example 8. \end{example}


\begin{thebibliography}{100}

\large


\bibitem{AP} H. Albayrak, S. Pehlivan, Statistical convergence and statistical continuity on locally solid Riesz spaces, Topology Appl. {159} (2012), 1887-1893.

\bibitem{AA} A. Aydin, The statitically unbounded $\tau$-convergence in locally solid Riesz spaces, Turkish J. Math., {44} (2020) 949-956.

\bibitem{Aytar1} S. Aytar, The rough limit set and the core of a real sequence, Numer. Funct. Anal. Optim. {29} (3-4) (2008),  283-290.

\bibitem{Aytar2} S. Aytar, Rough statistical convergence, Numer. Funct. Anal. Optim. {29} (3-4) (2008), 291-303.

\bibitem{Das} P. Das, Some further results on ideal convergence in topological spaces, Topology Appl. {159} (10-11) (2012), 2621-2626.

\bibitem{DGGS}  P. Das, S. Ghosal, A. Ghosh, S. Som, Characterization of rough weighted statistical limit set, Math. Slovaca {68} (4) (2018),  881-896.


\bibitem{Dundar2016}  E. D$\ddot{\mbox{u}}$ndar,  On rough $\mathcal{I}_2$-convergence of double sequences, Numer. Funct. Anal. Optim. 37 (4) (2016), 480-491.

\bibitem{Fast}  H. Fast,  Sur la convergence statistique, Colloq. Math. {2} (1951), 241-244.

\bibitem{F1}  J. A. Fridy,  On statistical convergence, Analysis {5} (4) (1985), 301-313.

\bibitem{F2}  J. A. Fridy,  Statistical limit points, Proc. Amer. Math. Soc. {118} (4) (1993), 1187-1192.

\bibitem{Ghosal2020} S. Ghosal, A. Ghosh, Rough weighted $\mathcal{I}$-limit points and weighted $\mathcal{I}$-cluster points in $\theta$-metric space, Math. Slovaca {70} (3) (2020)  667-680.

\bibitem{Hong}  L. Hong,  On order bounded subsets of locally solid Riesz spaces, Quaest. Math. {39} (3) (2016), 381-389.

\bibitem{Karakaya} V. Karakaya,  T. A. Chishti, Weighted statistical convergence. Iranian J. Sci. Technol. Trans. A  {33} (A3) (2009) 219-223.


\bibitem{KMSS0} P. Kostyrko, M. Ma\v{c}aj, T. \v{S}al\'{a}t, M. Sleziak,  $\mathcal{I}$-convergence and extremal limit points,  Math Slovaca {55} (4) (2005) 443-464.


\bibitem{KMSS}  P. Kostyrko,  M. Ma\v{c}aj,  T. \v{S}al\'{a}t,  O. Strauch,  On statistical limit points, Proc. Amer. Math. Soc. {129} (9) (2001), 2647-2654.

\bibitem{KSW}  P. Kostyrko,  T. \v{S}al\'{a}t,  W. Wilczy$\acute{\mbox{n}}$ski,  $\mathcal{I}$-convergence, Real Anal. Exchange, {30} (2000/2001), 669-685.


\bibitem{LR2} M. C. List$\acute{\mbox{a}}$n-Garc\'ia, F. Rambla-Barreno, Rough convergence and Chebyshev centers in Banach spaces, Numer. Funct. Anal. Optim. 35 (4) (2014), 432-442.

\bibitem{Luxemburg} W. A. J. Luxemburg, A. C. Zaanen, Riesz Space-I, North Holland, Amsterdam, 1971.

\bibitem{Pal} S. K. Pal, D. Chandra, S. Dutta,  Rough ideal convergence, Hacet. J. Math. Stat. {42} (6) (2013), 633-640.

\bibitem{Phu1} H. X. Phu, Rough convergence in normed linear spaces, Numer. Funct. Anal. Optim. {22} (1-2) (2001), 199-222.

\bibitem{Phu2}  H. X. Phu, Rough convergence in infinite dimensional normed space, Numer. Funct. Anal. Optim. {24} (3-4) (2003), 285-301.

\bibitem{Riesz}  F. Riesz, Sur la decomposition des operations functionelles lineaires, Alti del Congr. Internaz. del Mat. Bologna 1928, 3, Zanichelli (1930) 143-148.

\bibitem{Roberts} G. T. Roberts,  Topologies in vector lattices, Math. Proc. Cambridge Philos. Soc.  {48} (1952), 533-546.


\bibitem{Salat} T. \v{S}al\'{a}t,  On statistically convergent sequences of real numbers, Math. Slovaca {30} (2) (1980), 139-150.

\bibitem{Savas0}  E. Sava\c{s}, On generalized double statistical convergence in locally solid Riesz spaces, Miskolc Math. Notes {17}(1) (2016), 591-603.

\bibitem{Savas} E. Savas, P. Das,  A generalized statistical convergence via ideals, Appl. Math. lett. {24} (2011), 826-830.

\bibitem{SM}  E. Sava\c{s}, M. Et, On $(\Delta_{\lambda}^m,\mathcal{I})$-statistical convergence of order $\alpha,$  Period. Math. Hungar.  71 (2015) 135-145.


\bibitem{St}  H. Steinhaus, Sur la convergence ordinaire et la convergence asymptotique. Colloq. Math. {2} (1951), 73-74.



\end{thebibliography}
\end{document}